\documentclass[12pt,psamsfonts]{amsart}
\usepackage{amsmath}
\usepackage{amsthm}
\usepackage{amssymb}
\usepackage{amscd}
\usepackage{amsfonts}
\usepackage{amsbsy}
\usepackage{color}
\usepackage{graphicx,psfrag}
\usepackage{graphics}

\newcommand{\R}{\ensuremath{\mathbb{R}}}

\newcommand{\C}{\ensuremath{\mathbb{C}}}

\newtheorem {theorem} {Theorem}

\begin{document}

\title[First integrals for $n$-dimensional Lotka-Volterra systems]
{First integrals of a class of $n$-dimensional Lotka-Volterra
differential systems}

\author[J. Llibre, A.C. Murza and A.E. Teruel]
{Jaume Llibre, Adrian C. Murza and Antonio E. Teruel}

\address{Jaume Llibre, Departament de Matem\`atiques, Universitat
Aut\`onoma de Barcelona, 08193 Bellaterra, Barcelona, Catalonia,
Spain} \email{jllibre@mat.uab.cat}

\address{Adrian C. Murza, Institute of Mathematics ``Simion Stoilow''
of the Romanian Academy, Calea Grivi\c tei 21, 010702 Bucharest,
Romania} \email{adrian\_murza@hotmail.com}

\address{Antonio E. Teruel, Departament de Matem\`{a}tiques i Inform\`atica,
Universitat de les Illes Balears, Crta. de Valldemossa km. 7.5,
07122 Palma de Mallorca, Spain} \email{antonioe.teruel@uib.es}



\begin{abstract}
Lotka-Volterra model is one of the most popular in biochemistry. It
is used to analyze cooperativity, autocatalysis, synchronization at
large scale and especially oscillatory behavior in biomolecular
interactions. These phenomena are in close relationship with the
existence of first integrals in this model. In this paper we
determine the independent first integrals of a family of
$n$--dimensional Lotka-Volterra systems. We prove that when $n=3$
and $n=4$ the system is completely integrable. When $n\geq6$ is
even, there are three independent first integrals, while when
$n\geq5$ is odd there exist only two independent first integrals. In
each of these mentioned cases we identify in the parameter space the
conditions for the existence of Darboux first integrals. We also
provide the explicit expressions of these first integrals.
\end{abstract}

\maketitle

\section{Introduction and formulation of the problem}

The real nonlinear ordinary differential systems are widely used to
model processes or reactions in a variety of fields of science, from
biology and chemistry to economy, physics and engineering. The
qualitative theory of dynamical systems is employed to analyze the
behavior of these dynamical systems. Within this analysis one of the
important features is the existence of first integrals of the
differential systems defined in $\R^n$. This is mainly due to the
fact that the existence of a first integral allows to reduce the
dimension of the system by one. So in the qualitative theory of the
differential systems are important the methods allowing to detect
the presence of first integrals.

In this paper we shall apply the Darboux theory of integrability to
real polynomial Lotka-Volterra differential systems. This theory
provides a method of constructing first integrals of polynomial
differential systems, based on the number of invariant algebraic
hypersurfaces that they have. Since its publication in 1878, this
theory originally developed by Darboux, has been extended and/or
refined by many authors first for polynomial differential systems in
$\R^2$ see for instance \cite{CLS, Ch, CL1, CL2, Da, DLA, Ll, LZ3},
and later on for polynomial differential systems in $\R^n$ see
\cite{Jo, LB, LM, CL00, LZ1, LZ2, LZ4, LZ}.

The Lotka-Volterra differential systems (see \cite{L20, V31}) also
called Kolmogorov differential systems (see \cite{Ko}) are used to
model a wide range of experimental processes \cite{AKJ08, CDL04,
DPW88, W75}. In biochemistry, for example the pioneering work of
Wyman \cite{W75} models the autocatalytic chemical reactions, called
by Di Cera et al. \cite{DPW88} a ``turning wheel'' of one-step
transitions of the macromolecule. Turning wheels have multiple
applications in biochemistry. For instance, enzyme kinetics
\cite{W75}, circadian clocks \cite{MH75} and genetic networks
\cite{AKJ08, CDL04} are just a few of them.

In \cite{DPW88} it has been proved that when the law of mass
conservation applies, the autocatalytic chemical reactions between
$x_i,i=1,\ldots,n$ are governed by the following $n$-parameter
family of nonlinear differential equations differential equations
\begin{equation}\label{e1}
\begin{array}{l}
\dot{x}_1=x_1(k_1 x_2 - k_n x_n)=P_1(x_1,\ldots,x_n)=x_1K_1(x_1,
\ldots,x_n),\vspace{0.2cm}\\
\dot{x}_2=x_2( k_2 x_3 -k_1 x_1)=P_2(x_1,\ldots,x_n)=x_2K_2(x_1,
\ldots,x_n),\vspace{0.2cm}\\
\hspace{5cm}\vdots\vspace{0.2cm}\\
\dot{x}_n=x_n(k_n x_1 -k_{n-1}
x_{n-1})=x_nK_n(x_1,\ldots,x_n)=P_n(x_1,\ldots,x_n),
\end{array}
\end{equation}
where the parameters satisfy $k_i\in\R\backslash\{0\}.$ This is only
one of the multiple examples of $n$-dimensional Lotka-Volterra
differential systems.

In the rest of the paper we will study the first integrals of the
differential system \eqref{e1} using the Darboux theory of
integrability.

Let
\begin{equation}\label{X}
X=\displaystyle{P_1\frac{\partial}{\partial
x_1}+P_2\frac{\partial}{\partial
x_2}+\ldots+P_n\frac{\partial}{\partial x_n}},
\end{equation}
be the vector field associated to system \eqref{e1}. Let $U$ be an
open and dense subset of $\R^n$. A {\it first integral} of system
\eqref{e1} is a non-constant function $H:U\to \R$ such that it is
constant on the solutions of system \eqref{e1}, i.e. $XH=0$ in the
points of $U$. Two first integrals $H_i:U\to \R$ for $i=1,2$ of
system \eqref{e1} are {\it independent} if their gradients $\nabla
H_1$ and $\nabla H_2$ are independent in all the points of $U$
except perhaps in a zero Lebesgue measure set of $U$. System
\eqref{e1} is {\it completely integrable} if there exist $n-1$
independent first integrals.

Our main result is the following one.

\begin{theorem}\label{t1}
For the Lotka-Volterra differential system \eqref{e1} with $k_i\ne
0$ for $i=1,2,\ldots,n$ the following statements hold.
\begin{itemize}
\item[(a)] $H=\displaystyle{\sum_{i=1}^nx_i}$ is a first integral.
So in particular for $n=2$ the system is completely integrable.

\bigskip

\item[(b)] For $n\geq 3$ odd
\begin{equation*}\label{e3}
H_1=\displaystyle{\sum_{i=1}^nx_i},\quad
H_2=\displaystyle{x_1\prod_{i=2}^nx_i^{\mu_i}},
\end{equation*}
where
\begin{equation*}\label{lambdan1}
\mu_j=\displaystyle{\frac{k_1k_3\dots k_{j-2}}{k_2k_4\dots
k_{j-1}}}\,\, \mbox{if $j\ge 3$ odd}, \,\,
\mu_j=\displaystyle{\frac{k_{j+1}k_{j+3}\dots k_{n}}{k_jk_{j+2}\dots
k_{n-1}}}\,\, \mbox{if $j\ge 2$ even},
\end{equation*}
are two independent first integrals.

\bigskip

\item[(c)] For $n=3$ the system is completely integrable with the two
independent first integrals
\begin{equation*}\label{t1fi}
H_1=x_1+x_2+x_3,\qquad
H_2=\displaystyle{x_1x_2^{\frac{k_3}{k_2}}x_3^{\frac{k_1}{k_2}}}.
\end{equation*}

\bigskip

\item[(d)] If $n\geq 4$ is even and
$k_1k_3\cdots k_{n-1}=k_2k_4\cdots k_n$, then
\begin{equation*}\label{e4}
H_1=\displaystyle{\sum_{i=1}^nx_i},\quad H_2=\displaystyle{x_1
x_3^{\mu_3}\dots x_{n-1}^{\mu_{n-1}}},\quad H_3=\displaystyle{x_2
x_4^{\mu_4} \dots x_n^{\mu_n}},
\end{equation*}
where
\begin{equation*}\label{lambdan}
\mu_j=\displaystyle{\frac{k_{j+1}k_{j+3}\dots k_{n}}{k_jk_{j+2}\dots
k_{n-1}}}\,\, \mbox{if $j\ge 3$ odd}, \quad
\mu_j=\displaystyle{\frac{k_2k_4\dots k_{j-2}}{k_3k_5\dots
k_{j-1}}}\,\, \mbox{if $j\ge 4$ even},
\end{equation*}
are three independent first integrals.

\bigskip

\item[(e)] For $n=4$ the system is completely integrable if
$k_1k_3=k_2k_4,$ with the three independent first integrals
\begin{equation*}\label{t3fi}
H_1=x_1+x_2+x_3+x_4,\quad
H_2=\displaystyle{x_1x_3^{\frac{k_4}{k_3}}},\quad
H_3=\displaystyle{x_2x_4^{\frac{k_2}{k_3}}}.
\end{equation*}

\end{itemize}
\end{theorem}

Theorem \ref{t1} is proved in the next section.

Let $U$ be an open and dense subset of $\R^n$. The function $M:U\to
\R$ is a {\it Jacobi multiplier} for the Lotka-Volterra differential
system \eqref{e1} if
\[
\sum_{i=1}^n \frac{\partial (MP_i)}{\partial x_i}=0.
\]
The so-called {\it Jacobi Theorem}, see Theorem 2.7 of \cite{G},
applied to our system \eqref{e1} says that if system \eqref{e1}
admits a Jacobi multiplier $M$ and $n-2$ independent first
integrals, then the system admits an extra first integral. An easy
computation shows that the function
$$
M(x_1,\ldots,x_n)=\displaystyle{\frac{1}{\prod_{i=1}^nx_i}}
$$
is a Jacobi multiplier of system \eqref{e1}. However we cannot use
it to improve the number of independent first integrals, because
when the dimension is $n=3$ and $n=4$ the system is completely
integrable, while for $n\geq5$ we do not know $n-2$ independent
first integrals \cite{MT}.

\section{Proof of Theorem \ref{t1}}

Let $f\in \mathbb{R}[x_1,x_2,\ldots,x_{n}]$ be a polynomial. The
algebraic hypersurface $f=0$ of $\R^n$ is an {\it invariant
algebraic hypersurface} of the system \eqref{e1} if there exists a
polynomial $K\in \mathbb{R}[x_1,x_2,\ldots,x_{n}]$ such that
$Xf=Kf.$ The polynomial $K$ is called the {\it cofactor} of $f$. We
note that an invariant hypersurface $f=0$ has the property that if
an orbit of system \eqref{e1} has a point in $f=0$, then the whole
orbit is contained in $f=0$, for more details see for instance
Chapter 8 of \cite{DLA}.

{From} the definition of invariant algebraic hypersurface it follows
immediately that for $i=1,\ldots,n$ the hyperplanes $x_i=0$ are
invariant hyperplanes of system \eqref{e1}, and their corresponding
cofactors are $K_i(x_1,x_2,\ldots,x_{n})$.

The following result is due to Darboux, see \cite{Da}, or Chapter 8
of \cite{DLA}.

\begin{theorem}\label{t2}
Suppose that the polynomial vector field \eqref{e1} admits $n$
invariant algebraic surfaces $f_i=0$ with cofactors $K_i$ for
$i=1,2,\ldots,n.$ If there exist $\lambda_i \in \mathbb{R}$ not all
zero such that $\displaystyle{\sum_{i=1}^n \lambda_i K_i=0},$ then
the function $f_1^{\lambda_1}f_2^{\lambda_2}\ldots f_n^{\lambda_n}$
is a first integral of the vector field \eqref{e1}.
\end{theorem}

\begin{proof}[Proof of statement $(a)$ of Theorem \ref{t1}]
Let $H=\displaystyle{\sum_{i=1}^nx_i}$. Then an easy calculation
shows that $XH=0,$ where $X$ has been defined in \eqref{X}.
\end{proof}

\begin{proof}[Proof of statement $(b)$ of Theorem \ref{t1}]
Assume that $n\ge 3$ is odd. From statement $(a)$ of Theorem
\ref{t1} the function $H_1$ is a first integral. Now we calculate
the other first integral.

For $i=1,2,\ldots,n$ we know that the hyperplane $x_i=0$ is
invariant for system \eqref{e1}, and that its cofactor is
$K_i(x_1,x_2,\ldots,x_{n})= k_i x_{i+1}-k_{i-1}x_{i-1}$. From
Theorem \ref{t2} if there exist $\lambda_i$ not all zero and such
that $\displaystyle{\sum_{i=1}^n \lambda_i K_i=0},$ then
$H=x_1^{\lambda_1}x_2^{\lambda_2}\dots x_n^{\lambda_n}$ is a first
integral of system \eqref{e1}. Then we have
\begin{equation*}\label{e5}
\begin{array}{rl}
\displaystyle{\sum_{i=1}^n \lambda_i K_i}=& x_1
\displaystyle{\left(k_n\lambda_n-k_1\lambda_2\right)}+
x_2\displaystyle{\left(k_1\lambda_1-k_2\lambda_3\right)}+\vspace{0.2cm}\\
& x_3\displaystyle{\left(k_2\lambda_2-k_3\lambda_4\right)}+
x_4\displaystyle{\left(k_3\lambda_3-k_4\lambda_5\right)}+\vspace{0.2cm}\\
& \hspace{3cm}\vdots\\
&\displaystyle{x_{n-1}\left(k_{n-2}\lambda_{n-2}-k_{n-1}\lambda_n\right)}
+x_n\displaystyle{(k_{n-1}\lambda_{n-1}-k_n\lambda_1)}
\vspace{0.2cm}\\
=&0,
\end{array}
\end{equation*}
or equivalently
\begin{equation}\label{e6}
\begin{array}{l}
k_n\lambda_n-k_1\lambda_2= k_1\lambda_1-k_2\lambda_3=
k_2\lambda_2-k_3\lambda_4=
k_3\lambda_3-k_4\lambda_5= \vspace{0.2cm}\\
\hspace{6cm}\vdots\\
k_{n-2}\lambda_{n-2}-k_{n-1}\lambda_n=
k_{n-1}\lambda_{n-1}-k_n\lambda_1=0.
\end{array}
\end{equation}

Then it is easy to check that the solutions $\lambda_j$'s of system
\eqref{e6} are
\begin{equation*}
\lambda_j=\displaystyle{\frac{k_1k_3\dots k_{j-2}}{k_2k_4\dots
k_{j-1}} \lambda_{1}}\,\, \mbox{if $j\ge 3$ odd}, \,
\lambda_j=\displaystyle{\frac{k_{j+1}k_{j+3}\dots
k_{n}}{k_jk_{j+2}\dots k_{n-1}}\lambda_{1}}\,\, \mbox{if $j\ge 2$
even}.
\end{equation*}

Since the unique free lambda is $\lambda_1$, by Theorem \ref{t2} we
can choose $\lambda_1=1$ for obtaining the first integral $H_2$ of
system \eqref{e1} given in statement (b).

Clearly that the integrals $H_1$ and $H_2$ are independent because
the gradient $\nabla H_1=(1,1,\ldots,1)$ is independent of the
gradient $\nabla H_2$.
\end{proof}

\begin{proof}[Proof of statement $(c)$ of Theorem \ref{t1}]
Statement (c) follows immediately from statement (b).
\end{proof}

\begin{proof}[Proof of statement $(d)$ of Theorem \ref{t1}]
Assume that $n\ge 4$ even.  Now we calculate the two additional
first integrals to the integral $H_1$.

Taking into account that $k_1k_3\cdots k_{n-1}=k_2k_4\cdots k_n$ it
is easy to check that the solutions $\lambda_j$'s of system
\eqref{e6} can be written as
\begin{equation*}
\lambda_j=\displaystyle{\frac{k_{j+1}k_{j+3}\dots
k_{n}}{k_jk_{j+2}\dots k_{n-1}}\lambda_{1}}\,\, \mbox{if $j\ge 3$
odd}, \, \lambda_j=\displaystyle{\frac{k_2k_4\dots
k_{j-2}}{k_3k_5\dots k_{j-1}} \lambda_{2}}\,\, \mbox{if $j\ge 4$
even}.
\end{equation*}

Since the unique free lambdas are $\lambda_1$ and $\lambda_2$, we
can choose the following two choices: $(\lambda_{1},\lambda_2)=
(1,0)$ and $(\lambda_{1}, \lambda_2)=(0,1)$, and by applying Theorem
\ref{t2} we obtain the two independent first integrals $H_2$ and
$H_3$ given in the statement (d).

Clearly that these third integrals are independent since $H_2$ has
only even coordinates, $H_3$ only odd coordinates, and the
combination of the gradient vectors of $H_2$ and $H_3$ cannot
provide the gradient of $H_1$.
\end{proof}

\begin{proof}[Proof of statement $(e)$ of Theorem \ref{t1}]
Statement (e) follows immediately from statement (d).
\end{proof}

\section*{Acknowledgements}

The first author is partially supported by a FEDER-MINECO grant
MTM2016-77278-P, a MINECO grant MTM2013-40998-P, and an AGAUR grant
2014SGR-568. The second author acknowledges partial support from a
grant of the Romanian National Authority for Scientific Research and
Innovation, CNCS-UEFISCDI, project number PN-II-RU-TE-2014-4-0657.

\end{document}